\documentclass[11pt]{article}
\usepackage{amsmath,amssymb,amsthm}
\usepackage{graphicx,subfigure,float,url}
\usepackage{pdfsync}
\usepackage[english]{babel}
\usepackage[utf8]{inputenc}

\usepackage{fancybox}
\usepackage{listings}
\usepackage{mathtools}
\usepackage{amsfonts}
\usepackage[bottom]{footmisc}
\usepackage{verbatim}

\usepackage{float}
\usepackage{makeidx}

\normalbaselines

\newtheorem{Theorem}{Theorem}[section]

\newtheorem{Definition}[Theorem]{Definition}

\newtheorem{Proposition}[Theorem]{Proposition}

\newtheorem{Lemma}[Theorem]{Lemma}

\newtheorem{Corollary}[Theorem]{Corollary}

\newtheorem{Remark}[Theorem]{Remark}

\newcommand{\N}{\mathbb N}

\newcommand{\RR}{{{\rm I} \kern -.15em {\rm R} }}

\newcommand{\C}{{{\rm l} \kern -.42em {\rm C} }}

\newcommand{\nat}{{{\rm I} \kern -.15em {\rm N} }}

\newcommand{\be}{\begin{equation}}
\newcommand{\ee}{\end{equation}}
\newcommand{\beq}{\begin{eqnarray}}
\newcommand{\eeq}{\end{eqnarray}}
\newcommand{\beqs}{\begin{eqnarray*}}
\newcommand{\eeqs}{\end{eqnarray*}}
\newcommand{\bt}{\begin{Theorem}}
\newcommand{\et}{\end{Theorem}}
\newcommand{\br}{\begin{Remark}}
\newcommand{\er}{\end{Remark}}
\newcommand{\bc}{\begin{Corollary}}
\newcommand{\ec}{\end{Corollary}}
\newcommand{\bl}{\begin{Lemma}}
\newcommand{\el}{\end{Lemma}}
\newcommand{\bd}{\begin{definition}}
\newcommand{\ed}{\end{definition}}

\renewcommand{\geq}{\geqslant}
\renewcommand{\leq}{\leqslant}

\title{Asymptotic analysis of a Cucker-Smale system \\ with leadership and distributed delay}
\author{
Cristina Pignotti\footnote{Dipartimento di Ingegneria e Scienze dell'Informazione e Matematica, Universit\`{a} di L'Aquila, Via Vetoio, Loc. Coppito, 67010 L'Aquila Italy, \texttt{pignotti@univaq.it}.}
\and Irene Reche Vallejo\footnote{\texttt{irenereche92@gmail.com}.}
}

\date{}

\begin{document}

\textwidth=160 mm

\textheight=225mm

\parindent=8mm

\frenchspacing

\maketitle

\begin{abstract}
We extend the analysis developed in \cite{PRV} in order to prove convergence to consensus results for a Cucker--Smale type model with hierarchical leadership and distributed delay. Flocking estimates are obtained for a general interaction potential with divergent tail. We analyze also the model when the ultimate leader can change its velocity. In this case we give a flocking result under suitable conditions on the leader's acceleration.
\end{abstract}

\vspace{5 mm}

\def\qed{\hbox{\hskip 6pt\vrule width6pt
height7pt
depth1pt  \hskip1pt}\bigskip}



\section{Introduction}
\label{pbform}

\setcounter{equation}{0}

The celebrated Cucker-Smale model has been introduced in \cite{CS1, CS2} as a model for flocking, namely for phenomena where autonomous agents reach a consensus based on limited environmental information.
Let us consider $N\in \N$ agents and let $(x_i(t), v_i(t))\in\RR^{2d},$ $i=1,\dots, N,$ be their phase-space coordinates. As usual $x_i(t)$ denotes the position of the $i^{\textrm{th}}$ agent and $v_i(t)$ the velocity.
The Cucker-Smale model reads, for $t>0,$
\begin{equation}\label{originary}
\begin{split}
\dot x_i(t)&= v_i(t),\\
\dot v_i(t)&=\sum_{j=1}^N \psi_{ij}(t)(v_j(t)-v_i(t)),\qquad i=1,\dots, N,
\end{split}
\end{equation}
where  the communication rates $\psi_{ij}(t)$ are of the form
\begin{equation}\label{potential}
\psi_{ij}(t)=\psi (\vert x_i(t)-x_j(t)\vert )\,,
\end{equation}
being
$\psi: [0, +\infty)\rightarrow (0, +\infty )$ a suitable non-increasing potential functional.

\begin{Definition}{\rm
We say that a solution of (\ref{originary}) converges to consensus (or flocking) if
\begin{equation}\label{cons}
\sup_{t>0} \vert x_i(t)-x_j(t)\vert <+\infty\quad\quad\mbox {\rm and}\quad \quad \lim_{t\rightarrow +\infty} \vert v_i(t)-v_j(t) \vert =0\,,\quad \forall \ i,j=1,\dots, N.
\end{equation}
}\end{Definition}

The potential function considered by Cucker and Smale in \cite{CS1, CS2} is $\psi(s)=\frac 1
{(1+s^2)^\beta }$ with $\beta\geq 0$. They proved that there  is
unconditional convergence to flocking whenever
 $\beta <1/2$.
In the case $\beta\geq 1/2$, they obtained a conditional flocking result, namely convergence to flocking under appropriate assumptions on the initial data. Actually,  unconditional flocking can be obtained also
for $\beta =1/2$ (see e.g. \cite{HL}).

The extension of the flocking result to cover the case of non symmetric communication rates is due to Motsch and Tadmor \cite{MT}. Other variants and generalizations have been proposed, e.g. more general interaction  potentials, cone-vision constraints, leadership (see e.g. \cite{Couzin, CuckerDong, HaSlemrod, Mech, MT_SIREV, Shen, Vicsek, Yates}), stochastic terms (\cite{CuckerMordecki, Delay2, HaLee}), pedestrian crowds (see \cite{Cristiani, Lemercier}), infinito-dimensional kinetic models  (see \cite{Albi, Bellomo, canuto, carfor, degond, HT, Toscani}) and control models
(see \cite{Borzi, Caponigro1, Caponigro2, PRT, Wongkaew}).

Here, we consider the Cucker-Smale system with hierarchical leadership introduced by
Shen \cite{Shen}. In this model
the agents are ordered in a specific way, depending on which other agents they are leaders of or led by.
This reflects natural situations, e.g. in animals groups, where some agents are more  influential
than the others. We also add a distributed delay term, namely we assume that the agent $i$ adjusts its velocity depending on the information received from other agents on a time interval $[t-\tau , t].$ Indeed, it is natural to assume that there is a time delay in the information's transmission from an agent to the others.
The case of CS-model with hierarchical leadership and a pointwise time delay has been recently studied by the authors (\cite{PRV}). Other
models with (pointwise) time delay, without leadership, have been considered in \cite{Choi, ChoiLi, Delay1, PT}, while  for other extensions of Shen's results, without delay, we refer to \cite{Da, Li, LiY, LX, LHX}.

In order to present  our model, we first  recall some definitions from \cite{Shen}.

\begin{Definition}
The \textbf{leader set} $\mathcal{L}(i)$ of an agent $i$ in a flock $[1, 2,\dots, N]$ is the subgroup of agents that directly influence agent $i$, i.e. $\mathcal{L}(i)=\{j\;|\;\psi_{ij}>0\}$.
\end{Definition}

The Cucker-Smale system considered by Shen is then, for all $i\in\{1, \dots, N\}$ and $t>0$,

\begin{equation}\label{CSShen}
\begin{array}{l}
\displaystyle{
\frac{d x_i}{dt}={v_i},}\\
\displaystyle{
\frac{d {v_i}}{dt}=\sum_{j\in\mathcal{L}(i)}\psi_{ij}(t)({v_j}-{v_i}).}
\end{array}
\end{equation}
The interaction potential was analogous to the one of Cucker and Smale's papers and Shen proved convergence to consensus for $\beta <1/2.$

\begin{Definition}
A flock $[1, \dots, N]$ is an \textbf{HL-flock}, namely a flock under hierarchical leadership, if the agents can be ordered in such a way that:
\begin{enumerate}
	\item if $\psi_{ij}\neq0$ then $j<i$, and
	\item for all $i>1$, $\mathcal{L}(i)\neq \emptyset$.
\end{enumerate}
\end{Definition}

\begin{Definition}
For each agent $i=1,\dots, N,$
we define
the
$m$-th level leaders of $i$ as
$$\mathcal{L}^0(i)=\{i\},\;\; \mathcal{L}^1(i)=\mathcal{L}(i),\;\; \mathcal{L}^2(i)=\mathcal{L}(\mathcal{L}(i)),\;\; \dots,\;\; \mathcal{L}^m=\mathcal{L}(\mathcal{L}^{m-1}(i)),\;\; \dots$$
for $m\in\N$, and denote the set of all leaders of the agent $i$, direct or indirect, as
$$[\mathcal{L}](i)=\mathcal{L}^0(i)\;\cup\; \mathcal{L}^1(i)\;\cup\; \dots$$
\end{Definition}
For a fixed positive time $\tau$ and for every $t>0,$ our system is the  following:

\begin{equation}\label{CSShendelay}
\begin{array}{l}
\displaystyle{\frac{d {x_i}}{dt}(t)={v_i}(t),}\\
\displaystyle{
\frac{d {v_i}}{dt}(t)=\sum_{j\in\mathcal{L}(i)}\int_{t-\tau }^t\mu (t-s) \psi_{ij}(s )[{v_j}(s)-{v_i}(t)]\, ds ,}
\end{array}
\end{equation}
for all $i\in\{1, \dots, N\},$
with initial conditions, for $s\in [-\tau, 0],$
\begin{equation}\label{IC}
\begin{array}{l}
{x_i}(s)=x_i^0(s),\\
{v_i}(s)=v_i^0(s),
\end{array}
\end{equation}
for some continuous functions $x_i^0$ and $v_i^0,$ $i=1,\dots, N.$
The communication rates are
$$\psi_{ij}(t)=\psi(|x_i(t)-{x_j}(t)|)$$
for some
non-increasing, nonnegative, continuous interaction potential $\psi.$
The weight function $\mu : [0,\tau]\rightarrow \RR$ is assumed to be bounded and nonnegative,
with
\begin{equation}\label{boundmu}
\int_0^\tau \mu (s) ds = \mu_0 >0.
\end{equation}
We will prove a flocking result under the assumption
\begin{equation}\label{divergent}
\int_0^{+\infty} \psi(s) ds =+\infty\,.
\end{equation}
Then, our result extends and generalizes the one of Shen.
Note that in \cite{PRV} we have proved a flocking result in the case of a pointwise time delay.
We can formally obtain  the model studied in \cite{PRV}  if the weight $\mu(\cdot)$ is a Dirac delta function  centered at $t=\tau\,.$

The paper is organized as follows.
In section \ref{sect2} we give some preliminary properties of system (\ref{CSShendelay}), in particular we prove the positivity and boundedness properties for the velocities.
In section \ref{sect3} we will prove the flocking result for the system (\ref{CSShendelay}).
Finally, in section \ref{sect4} we will consider the model under hierarchical leadership and a free--will leader and we will prove flocking estimates under suitable growth assumptions on the acceleration of the free--will leader.

\section{Preliminary properties} \label{sect2}
\setcounter{equation}{0}

Before proving our main result, namely the convergence to consensus thorem, we need
some general properties of the Cucker-Smale model (\ref{CSShendelay}), such as the positivity property and the boundedness of the velocities.
The following propositions extend analogous results of \cite{Shen}.

\begin{Proposition}\label{positivity}
Let
us  consider the system of scalar  equations
\begin{equation}\label{positivitywithdelay}
	\begin{array}{l}
	\displaystyle{
\frac{du_i}{dt}(t)=\sum_{j\in\mathcal{L}(i)}\int_{t-\tau}^t\mu (t-s) \psi_{ij}(s)[u_j(s)-u_i(t)]\, ds, \;\;\; i=1,\dots,N,\; t>0,}\\
\displaystyle{
u_i(s)=u_i^0(s), \;\;\; i=1,\dots, N,\;s\in[-\tau ,0],}
\end{array}
\end{equation}
where $u_i^0(\cdot),$ $i=1,\dots, N,$ are continuous functions.
If $u_i^0(s)\geq0$ for all $i=1, \dots, N$, and all $s\in[-\tau ,0]$, then $u_i(t)\geq0$ for all $i$ and $t>0$.
\end{Proposition}

\begin{proof}
Observe that if an agent $j$ is in the leader set $[\mathcal{L}](i)$ of the agent $i,$ then it is not influenced by agents outside of $[\mathcal{L}](i)$. Thus, it is sufficient to prove the statement for the system (\ref{positivitywithdelay}) restricted to the agents in $[\mathcal{L}](i)$, for each $i=1, \dots, N$.

We then  proceed by induction. Consider the first agent, i.e. agent 1. By definition of an HL-flock, $\mathcal{L}(1)=\emptyset$, which gives

\begin{equation}
	\frac{d u_1}{dt}=0 \ \mbox{\rm and so} \quad u_1(t)=u_1(0)=u_1^0(0)\geq 0,\quad \forall\  t\geq0.
	\label{u1}
\end{equation}
Using (\ref{u1}), the equation for the agent 2 becomes

$$\frac{d u_2}{dt}(t)=\int_{t-\tau}^t \mu (t-s) \psi_{21}(s)[u_1(s)-u_2(t)] ds =(u_1(0)-u_2(t))
\int_{t-\tau}^t \mu (t-s) \psi_{21}(s) ds\,
.$$
Arguing by contradiction, we assume that $u_2(\bar{t})<0$ for some $\bar{t}>0.$ Then, let
us denote

$$t^*=\inf\{t>0\;|\;u_2(s)<0\ \mbox{\rm for}\ s\in (t,\bar {t})\,\}.$$
Hence, by definition of $t^*$, $u_2(t^*)=0$ and
$u_2(s)<0$ for $s\in (t^*,\bar{t}).$
So, using again (\ref{u1}),

$$\frac{d u_2}{dt}(t)=({u_1(0)}-{u_2(t)})\int_{t-\tau }^t \mu (t-s) \psi_{21} (s)ds\geq 0,\quad t\in [t^*, \bar {t}),$$
which is in contradiction with  $u_2(t)<0$ for $t\in (t^*, \bar t)$ and $u_2(t^*)=0.$ This ensures  that $u_2(t)\geq 0$ for all $t\ge 0$.

Now, as the induction hypothesis, assume that $u_i(t)\geq 0$ for all $t>0$ and for all $i\in\{1,\dots, k-1\}$.

The equation for agent $k$ is

$$\frac{d u_k}{dt}(t)=\sum_{j\in\mathcal{L}(k)}\int_{t-\tau}^t \mu(t-s) \psi_{kj}(s)[u_j(s)-u_k(t)] ds\,,\ \ t>0.$$

As in the first step, let us assume by contradiction that $u_k(\bar{t})<0$ for some $\bar{t}>0$ and let
us denote

$$t^*=\inf\{t>0\;|\; u_k(s)<0\ \mbox{\rm for}\ s\in (t,\bar {t})\,\}.$$

Then,  $u_k(t^*)=0$ and $u_k(s)<0$ for $s\in (t^*, \bar{t}).$ We can use the induction hypothesis on the agents $j\in\mathcal{L}(k)\subseteq \{1, \dots, k-1\}$, so

$$\frac{d u_k}{dt}(t)=\sum_{j\in\mathcal{L}(k)}\int_{t-\tau}^t \mu (t-s) \psi_{kj}(s)[{u_j(s)}-{u_k(t)}] ds \geq 0,\quad t\in [t^*, \bar t)\,,$$
which gives a contradiction.

\noindent Therefore, we have proved that $u_i(t)\geq 0$ for all $i\in\{1, \dots, N\}.$
 \end{proof}
As in the undelayed case (see Th. 4.2 of \cite{Shen}) we can now deduce from the previous proposition the boundedness result for the velocities.

\begin{Proposition}\label{boundedness}
Let $\Omega$ be a convex and compact domain in $\RR^d$ and  let $({x_i},{v_i})$ be a solution of system $(\ref{CSShendelay})$. If ${v_i}(s)\in\Omega$ for all $i=1,\dots,N$ and $s\in[-\tau ,0]$, then ${v_i}(t)\in\Omega$ for all $i=1,\dots,N$ and $t>0$.
In particular, if $\Omega$ is the ball with center $0$ and radius
\begin{equation}\label{costanteD0}
D_0=\max_{1\leq i\leq N}\max_{s\in[-\tau ,0]}|{v_i}(s)|,
\end{equation}
 then $\vert {v_i}(t)|\leq D_0$ for all $t>0$ and $i=1,\dots, N$.
\end{Proposition}




\section{Convergence to consensus} \label{sect3}
\setcounter{equation}{0}

Here we will prove
the announced flocking result for the CS-model under hierarchical leadership with distributed delay (\ref{CSShendelay}).
Our proof extends to the model at hand the one in \cite{PRV}, with pointwise delay.
We need a preliminary lemma.
\begin{Lemma}\label{usoLyap}
Let $(x, v)$ be a trajectory in the phase--space, namely $\frac{dx}{dt}(t)=v(t)$ for $t\ge 0\,.$
Assume that
\begin{equation}\label{PP1}
\frac {d\vert v\vert}{dt}(t)\le -d_0\psi (\vert x(t)\vert + M) \vert v(t)\vert + ce^{-bt}\quad \forall \ t \ge t_0,
\end{equation}
for some nonnegative constants $M, c, t_0$ and $b, d_0>0,$ where $\psi: [0,+\infty)\rightarrow (0,+\infty)$ is a continuous function satisfying $(\ref{divergent}).$
Then, there exists a suitable positive constant $C$ such that
$$\vert x(t)\vert \le C, \quad t\ge 0\,.$$
\end{Lemma}

\begin{proof}
Let us consider the functionals (cfr. \cite{HL, PRV})
\begin{equation}\label{PP2}
{\mathcal {F}}_\pm(t)= \vert v(t)\vert \pm d_0\phi (\vert x(t)\vert + M),
\end{equation}
where $\phi $ is a primitive of $\psi,$ namely $\phi^\prime (s)=\psi (s),$ $s\in (0,+\infty)\,.$

From (\ref{PP1}) we deduce
 \begin{equation}\label{PP3}
\begin{array}{l}
\displaystyle{
\frac{d\mathcal{F}_{\pm}}{dt}(t)=
\frac{d|{v}|}{dt}(t)\pm d_0\psi(|{x(t)}|+M )\frac{d|{x}|}{dt}(t)
}\\
\hspace{1 cm}
\displaystyle{
\leq -d_0\psi(|{x}(t)|+M)|{v(t)}|\pm d_0\psi(|{x(t)}|+M)\frac{d|{x}|}{dt}(t)+ce^{-bt}}\\
\hspace{1 cm}\displaystyle{
=d_0\psi(|{x(t)}|+M)\left(
\pm\frac{d|{x}|}{dt}(t) -|{v(t)}| \right )+ce^{-bt}\leq ce^{-bt}\,,\quad t\ge t_0\,,}
\end{array}
\end{equation}
where we have used
\begin{equation}\label{PP4}
\left\vert \frac{d \vert x(t)\vert }{dt}\right\vert \le \vert v(t)\vert\,.
\end{equation}
Now, integrating (\ref{PP3}) on the time interval $[t_0,t],$ we obtain
$$\mathcal{F}_{\pm}(t)-\mathcal{F}_{\pm}(t_0 )\leq c\int_{t_0} ^te^{-bs}\;ds=\frac{c}{b}(e^{-bt_0}-e^{-bt})\leq \frac{c}{b},$$
which implies

$$|{v}(t)|-|{v}(t_0 )|\leq
\pm d_0\left( \phi\left(\vert x (t_0)\vert + {M}\right)-\phi\left(|{x}(t)|+{M}\right)\right)+\frac{c}{b},$$
namely
\begin{equation}\label{PP5}
 |{v}(t)|-|{v}(t_0 )|\leq -d_0\left|\int_{|{x}(t_0 )|+{M}}^{|{x}(t)|+{M}}\psi(s)\;ds\right| +\frac{c}{b}.
\end{equation}
In particular, from (\ref{PP5}), we deduce

\begin{equation}\label{PP6}
	|{v}(t_0 )|+\frac c b\geq d_0\left|\int_{|{x}(t_0)|+{M}}^{|{x}(t)|+{M}}\psi(s)\;ds\right|.
\end{equation}

Then, assumption (\ref{divergent}) ensures  the existence of a constant  $x_M>0$ such that

$$|{v}(t_0)|+\frac{c}{b}=d_0\int_{|{x}(t_0)|+{M}}^{x_M}\psi(s)\;ds ,$$
which, together with (\ref{PP6}), implies

$$|{x}(t)|\leq C,\quad\forall\  t\ge 0\,,$$
being $\psi$ is a nonnegative function.
\end{proof}

\begin{Theorem}\label{flockingDD}
Let $(x_i, v_i),$ $i=1, \dots, N,$ be a solution of the Cucker-Smale system under hierarchical leadership with distributed delay $(\ref{CSShendelay})$ with initial conditions $(\ref{IC}).$ Assume
that the potential function $\psi$ satisfies $(\ref{divergent}).$
Then,
\begin{equation}\label{flockingestimate}
\vert v_i(t) -v_j(t)\vert =O(e^{-Bt}),\quad \forall \ i, j=1,\dots, N,
\end{equation}
for  a suitable constant $B>0$ depending only on the initial configuration and the parameters of the system.
\end{Theorem}

\begin{proof}
We will use induction on the number of agents in the flock.
Consider first a flock of 2 agents $[1,2]$. Recall that, by definition of an HL-flock, $\mathcal{L}(2)\neq \emptyset$, i.e. $\psi_{21}>0$. Moreover, $\psi_{12}=0$. Then,
	
\begin{equation}\label{PPzero}
\frac{d{v_1}}{dt}=0 \quad \Rightarrow \quad {v_1}(t)= {v_1}(0),\quad\forall\ t>0,
\end{equation}
and
\begin{equation}\label{PP7}
\frac{d {v_2}}{dt}(t)=\int_{t-\tau}^t\mu (t-s)\psi_{21}(s)[{v_1}(s)-{v_2}(t)] ds=({v_1}(0)-{v_2}(t))\int_{t-\tau}^t\mu (t-s) \psi_{21}(s) ds,\quad t\ge\tau.
\end{equation}	
We now denote
\begin{equation}\label{PP8}
{y}_2(t)={x_2}(t)-{x_1}(t)\quad\mbox{\rm  and}\quad {w}_2(t)={v_2}(t)-{v_1}(t).
\end{equation}
Then, from (\ref{PP7}), we obtain

\begin{equation}\label{PP9}
\frac{d {w_2}}{dt}(t)=\frac{d {v_2}}{dt}(t)-{\frac{d {v_1}}{dt}}(t)=\int_{t-\tau}^t \mu (t-s) \psi_{21}(s)[v_1(s) -v_2(t)]ds,\quad t\ge\tau ,
\end{equation}
and thus, using also (\ref{PPzero}),
$$
\frac 12 \frac {d \vert w_2\vert^2}
{dt}(t)=-\vert w_2(t)\vert^2\int_{t-\tau}^t\mu (t-s) \psi_{21}(s) ds\,,
$$
which implies
\begin{equation}\label{PP10}
 \frac{d|{w_2}|}{dt}(t)\leq -\vert w_2(t)\vert \int_{t-\tau}^t\mu(t-s)\psi\left(\vert {x_2}(s)-
 {x_1}(s)\vert \right )ds\,, \quad t\ge\tau\,.
\end{equation}
Therefore, from (\ref{PP10}), we deduce that
 $|{w_2}(t)|$ is decreasing in time for $t\ge \tau\,.$
Now, observe that for $t>\tau $ and $s\in [t-\tau,t],$ we have
$$
\begin{array}{l}
\displaystyle{{x_1}(s)-{x_2}(s)={x_1}(t)-{x_2}(t)+\int_t^{s}{({x_1}-{x_2})^\prime }(\sigma)\;d\sigma}\\
\hspace{2 cm}\displaystyle{
={x_1}(t)-{x_2}(t)+\int_s^{t}{w_2}(\sigma )\;d\sigma,}
\end{array}
$$
which gives, recalling Lemma \ref{boundedness},

\begin{equation}\label{PP11}
|{x_1}(s)-{x_2}(s)|\leq |{x_1}(t)-{x_2}(t)|+2 D_0\tau=|{y_2}(t)|+2 D_0\tau\,,
\quad t\ge\tau\,,
\end{equation}
with $y_2(t), w_2(t)$ defined in (\ref{PP8}) and $D_0$ the bound on the initial velocities defined in (\ref{costanteD0}).

Using this inequality in (\ref{PP10}) and recalling that the potential function $\psi$ is not increasing, we obtain

\begin{equation}\label{PP12}	
	\frac{d|{w_2}|}{dt}(t)\leq -\vert w_2(t)\vert \int_{t-\tau}^t\mu (t-s)\psi(|{y_2}(t)|+2\tau D_0)ds= -\mu_0 \vert w_2(t)\vert \psi(|{y_2}(t)|+2\tau D_0)\,,\quad t\ge\tau\,,
\end{equation}
where $\mu_0$ is the positive constant in (\ref{boundmu}).
Then, the pair state-velocity $(y_2,w_2)$ satisfies the inequality (\ref{PP1}) with $t_0=\tau,$ $d=\mu_0,$
$M=2\tau D_0$ and $c=0\,.$ Therefore, we can apply Lemma \ref{usoLyap} obtaining
$\vert y_2(t)\vert \le C_2$ for some positive constant $C_2.$ So, for a suitable constant $y^2_M,$

\begin{equation}\label{PP13}
|{y_2}(t )|+2\tau D_0\leq y^2_M,
\quad t\ge\tau\,.
\end{equation}
Now, from (\ref{PP12}) and (\ref{PP13}) we deduce

$$\frac{d|{w_2(t)}|}{dt}\leq -\mu_0\psi(y^2_M)|{w_2}(t)|,\quad t\ge\tau\,,$$
and the Gronwall inequality implies
\begin{equation}\label{PP14}
|{w_2}(t)|\leq e^{-\mu_0\psi(y^2_M)(t-\tau ) }|{w_2}(\tau )|,\quad t\ge\tau\,.
\end{equation}
In order to complete our inductive step we will need also estimates on the distances $\vert v_i(s) -v_j(t)\vert$ and $\vert v_i(s) -v_j(s)\vert$ for $j=1,2$ and $s\in [t-\tau, t].$

Now, since ${v_1}(t)$ is constant for $t\ge\tau ,$ we easily deduce

\begin{equation}\label{PP15}
|{v_1}(s)-{v_2}(t)|=|{v_1}(t)-{v_2}(t)|=O(e^{-\psi(y^2_M)t}).
\end{equation}
Observe also that, for $s\in [t-\tau, t],$

\begin{equation}\label{PP16}
\begin{array}{l}
\displaystyle{
|{v_2}(s)-{v_2}(t)|= \left|\int_{s}^t {v_2}^\prime (\sigma)\;d\sigma \right|=\left|\int_{s}^t
\int_{\sigma -\tau}^{\sigma } \mu(\sigma-r)\psi_{21}(r)[{v_1}(r )-{v_2}(\sigma )] \;d r \;d\sigma \right|}\\
\hspace{1 cm}
\displaystyle{
\leq c\int_{s}^te^{-\psi(y^2_M)\sigma}\;d\sigma
\leq c\tau e^{-\psi(y^2_M)(t-\tau)}=c\tau e^{\psi(y^2_M)\tau}e^{-\psi(y^2_M)t}=O(e^{-\psi(y^2_M)t}).}
\end{array}
\end{equation}
Since

\begin{equation}\label{PP17}
|{v_2}(s)-{v_1}(t)|\leq |{v_2}(s)-{v_2}(t)|+|{v_2}(t)-{v_1}(t)|\,,
\end{equation}
from previous estimates we thus obtain
\begin{equation}\label{PP18}
|{v_2}(s)-{v_1}(t)|=O(e^{-\psi(y^2_M)t})\,,\quad t>\tau,\ s\in [t-\tau , t]\,.
\end{equation}
Moreover, of course,
$\vert v_1(s)-v_1(t)\vert = O(e^{-\psi(y^2_M)t}),$ being $v_1(t)$ constant for $t\ge\tau\,.$

We assume now, by induction, that analogous exponential estimates are satisfied
for a  flock of
 $l-1$ agents $[1, \dots, l-1]$ with $l>2$, i.e. there exists some constant $b >0$ such that, $\forall\ i, j\in\{1,\dots,l-1\},$

\begin{align}
	&|{v_i}(t)-{v_j}(t)|=O(e^{-bt}),\label{inductionassumptiondelay1}\\
	&|{v_i}(s)-{v_j}(t)|=O(e^{-bt}), \quad t>\tau, \ s\in [t-\tau , t].
	\label{inductionassumptiondelay2}
\end{align}

 Then, we want to prove that such  estimates hold true  also for a flock with $l>2$ agents $[1, \dots, l]$.
This will complete the proof.
For this aim, define the average position and velocity of the leaders of agent $l$,

\begin{equation}\label{PP19}
{\hat{x_l}}=\frac{1}{d_l}\sum_{i\in\mathcal{L}(l)}{x_i}(t)\;\;\; \mbox{and}\;\;\;{\hat{v_l}}=\frac{1}{d_l}\sum_{i\in\mathcal{L}(l)}{v_i}(t), \;\;\; d_l=\#\mathcal{L}(l).
\end{equation}
Also, define
\begin{equation}\label{PP20}
 {y_l(t)}={x_l(t)}-{\hat{x_l}(t)}\quad\mbox{\rm  and}\quad
{w_l(t)}={v_l(t)}-{\hat{v_l}(t)}.
\end{equation}
Then,

\begin{equation}\label{PP21}
\frac{d {w_l}}{dt}(t)=\frac{d {v_l}}{dt}(t)-\frac{d {\hat{v_l}}}{dt}(t)=\sum_{j\in\mathcal{L}(l)}\int_{t-\tau}^t\mu (t-s) \psi_{lj}(s)[{v_j}(s)-{v_l}(t)] ds -\frac{d {\hat{v_l}}}{dt}(t)\;.
\end{equation}

By adding and subtracting $\sum_{j\in\mathcal{L}(l)} \int_{t-\tau}^t\mu (t-s) \psi_{lj}(s)ds\; {\hat{v_l}(t)}$ in (\ref{PP21}) we get

\begin{equation}\label{PP22}	
\frac{d {w_l}}{dt}=-w_l(t)\sum_{j\in\mathcal{L}(l)}
\int_{t-\tau }^t \mu (t-s) \psi_{lj}(s)ds +
\sum_{j\in\mathcal{L}(l)}\int_{t-\tau}^t\mu (t-s) \psi_{lj}(s)[{v_j}(s)-{\hat {v_l}}(t)] ds
-\frac{d {\hat{v_l}}}{dt}.
\end{equation}
Using the induction hypothesis (\ref{inductionassumptiondelay2}), since $\mathcal{L}(i),\mathcal{L}(l)\subseteq [1, \dots, l-1]$,

\begin{equation}\label{PP23}
	\frac{d {\hat{v_l}}}{dt}=\frac{1}{d_l}\sum_{i\in\mathcal{L}(l)}\frac{d {v_i}}{dt}=\frac{1}{d_l}\sum_{i\in\mathcal{L}(l)}\sum_{j\in\mathcal{L}(i)}\int_{t-\tau}^t\mu (t-s)\psi_{ij}(s)[{v_j}(s)-{v_i}(t)]\;ds=O(e^{-bt}).
\end{equation}
Using again the induction hypothesis  (\ref{inductionassumptiondelay2}),

\begin{equation}\label{PP24}
\begin{array}{l}
\displaystyle{
	\sum_{j\in\mathcal{L}(l)}\int_{t-\tau}^t \mu (t-s) \psi_{lj}(s)[{v_j}(s)-{\hat{v_l}}(t)]\;ds}
\\
\displaystyle{\hspace{2 cm}
	= \frac 1 {d_l}\sum_{j\in\mathcal{L}(l)}\int_{t-\tau}^t\mu (t-s)\psi_{lj}(s)
	\Big (\sum_{i\in\mathcal{L}(l)} [v_j(s) -v_i(t)]\Big )\;ds
	=O(e^{-bt}).}
\end{array}
\end{equation}
So, identity (\ref{PP22}) can be rewritten as

\begin{equation} \label{PP25}
	\frac{d {w_l}}{dt}(t)=- w_l(t)\sum_{j\in\mathcal{L}(l)}\int_{t-\tau}^t\mu (t-s)\psi_{lj}(s)\;ds+ O(e^{-bt}),
\quad t\ge\tau\,.	
\end{equation}
with
$$\psi_{lj}(s)=\psi(|{x_l}(s)-{x_j}(s)|).$$
Observe that for every $j\in\mathcal{L}(l)$ it results
\begin{equation}\label{PP26}
\begin{array}{l}
\displaystyle{
|{x_l}(s)-{x_j}(s)|\le |{x_l}(s)-{\hat{x}_l}(s)\vert +|{x_j}(s)-{\hat{x}_l}(s)|}
\\
\hspace{2.6 cm}\displaystyle{
\leq |{y_l}(s)|+M_l,}
\end{array}
\end{equation}
for some positive $M_l$, due to the induction's assumption. Then, (\ref{PP25}) gives

\begin{equation}\label{PP27}
	\frac{d|{w_l}|}{dt}(t)\leq -d_l \vert w_l(t)\vert \;\int_{t-\tau}^t \mu (t-s) \psi\left(|{y_l}(s)|+M_l\right)\; ds+ce^{-bt},\quad t\ge\tau\,.
\end{equation}
Now, note that from Proposition \ref{boundedness}, $|{v_i}(t)|\leq D_0$ for all $i$ and for all $t>0$, which implies

$$|{w_l}(t)|\leq \frac{1}{d_l}\sum_{j\in\mathcal{L}(l)}|{v_j}(t)-{v_l}(t)|\leq \frac{1}{d_l}\sum_{j\in\mathcal{L}(l)}2D_0=2D_0.$$
Then,
\begin{equation}\label{PP28}
|{y_l}(s)|\leq |{y_l}(t)|+2\tau D_0,\quad t\ge \tau\,, \ s\in [t-\tau, t]\;,
\end{equation}
which used in (\ref{PP27}), recalling that $\psi$ in not increasing, yields

\begin{equation}\label{PP29}
	\frac{d|{w_l}|}{dt}(t)\leq -d_l \mu_0 \psi\left(|{y_l}(t)|+2\tau D_0+M_l\right)|{w_l}(t)|+ce^{-bt}.
\end{equation}

We can then apply Lemma \ref{usoLyap} to the pair state-velocity $(y_l, w_l)$ to conclude that
$\vert y_l (t)\vert \le C_l$ for some positive constant $C_l.$ So, for a suitable constant $y_M^l,$
$$\vert y_l(t)\vert +2\tau D_0+M_l\le y_M^l,\quad t\ge\tau\,.$$
Using the above estimate in (\ref{PP28}) we then obtain
$$\frac{d|{w_l}|}{dt}\leq -d_l\mu_0\psi(y^l_M)|{w_l}(t)|+ce^{-bt},$$
and therefore, from the
Gronwall's inequality we deduce,
\begin{equation}\label{Q1}
\vert w_l(t)\vert \le Ce^{-B^lt}\,,
\end{equation}
for suitable positive constants $C, B^l.$

Thus, from (\ref{Q1}) and the induction hypothesis (\ref{inductionassumptiondelay1}), for every $j\in {\mathcal L}(l),$ we have

\begin{equation}
	|{v_l}(t)-{v_j}(t)|\le \vert v_l(t)-\hat v_l(t)\vert +
	\vert \hat v_l(t) -v_j(t)\vert =O(e^{-Bt}).
	\label{Q2}
\end{equation}

Now, to complete the induction argument, we only have  to prove that, for all $t>0$ and $i,j\in\{1, \dots,l\}$,

\begin{equation}
|{v_i}(s)-{v_j}(t)|=O(e^{-Bt}),
\label{Q3}
\end{equation}
for a suitable positive constant $B.$

If $i,j\in\{1, \dots, l-1\}$, then (\ref{Q3}) is true by (\ref{inductionassumptiondelay2}). Let us consider the case $i\in\{1, \dots, l-1\}$ and $j=l$. Then,

$$
|{v_i}(s)-{v_l}(t)|
\leq
|{v_i}(s)-{v_i}(t)|+\vert {v_i}(t)-{v_l}(t)|=O(e^{-Bt}),$$
by (\ref{inductionassumptiondelay2}) and (\ref{Q2}), for suitable $B.$

\noindent Consider now $i=j=l$. Then, using previous estimates we see that

\begin{equation}
\label{Q4}
\begin{array}{l}
\displaystyle{
|{v_l}(s)-{v_l}(t)|=\left|\int_{s}^t {v_l}^\prime (\sigma )\;d\sigma\right|=\left|\int_{s}^t \sum_{k\in\mathcal{L}(l)}\int_{\sigma -\tau }^{\sigma} \mu (\sigma -r) \psi_{lj}(r)\left({{v_k}(r)-{v_l}(\sigma )}\right)\;d r\; d\sigma\right|}\\
\hspace{2 cm}\displaystyle{
\leq \bar{c}\int_{s}^te^{-B\sigma }\;d\sigma \le \bar{c}\tau e^{-B(t-\tau)}=\bar{c}\tau e^{B\tau}e^{-B t}=O(e^{-Bt}).}
\end{array}
\end{equation}
Also for the last case, where $j\in\{1, \dots, l-1\}$ and $i=l$, using (\ref{Q4}) we have

$$|{v_l}(s)-{v_j}(t)|\leq |{v_l}(s)-{v_l}(t)|+|{v_l}(t)-{v_j}(t)|=O(e^{-Bt}),$$
by the previous case and (\ref{Q2}). Then, we have proved that  (\ref{Q3}) is satisfied for all $i,j\in \{1,\dots,l\}$ and this concludes the proof of the theorem.
\end{proof}


\section{The case of free-will leader} \label{sect4}
\setcounter{equation}{0}

It may happen that the leader of the flock, instead of moving at a constant velocity, takes off or changes its rate in order to avoid a danger,  for instance due to the presence of predator species. Thus, it is important to consider this situation in the mathematical model.

The Cucker-Smale model with a free-will leader is, then,

\begin{equation}\label{agent1}
\begin{array}{l}
\displaystyle{\frac{d {x_1}}{dt}(t)={v_1}(t),}\\
\displaystyle{
\frac{d {v_1}}{dt}(t)=f(t),}
\end{array}
\end{equation}
where $f:[0,+\infty)\rightarrow \RR^d$ is a continuous integrable function, that is,
\begin{equation}\label{acceleration}
\Vert f\Vert_1=\int_0^{+\infty} \vert f(t)\vert \, dt <+\infty\,,
\end{equation}
for the motion of the free-will leader, and the Cucker-Smale model under hierarchical leadership and distributed delay, as in the previous sections, for the other agents,
 namely
\begin{equation}\label{CSShendelayFW}
\begin{array}{l}
\displaystyle{\frac{d {x_i}}{dt}(t)={v_i}(t),}\\
\displaystyle{
\frac{d {v_i}}{dt}(t)=\sum_{j\in\mathcal{L}(i)}\int_{t-\tau}^t\mu (t-s)\psi_{ij}(s )[{v_j}(s)-{v_i}(t)]\, ds,}
\end{array}
\end{equation}
for all $i\in\{2, \dots, N\}.$
The initial data are assigned, as usual, on the time interval $[-\tau , 0],$ i.e.
\begin{equation}\label{ICFW}
\begin{array}{l}
{x_i}(s)=x_i^0(s),\\
{v_i}(s)=v_i^0(s),
\end{array}
\end{equation}
for some continuous functions $x_i^0$ and $v_i^0,$ for $i=1,\dots ,N.$

The flocking result below
 extends the one proved by Shen \cite{Shen} for the undelayed case. The case with pointwise delay  has been studied in \cite{PRV}. Here, we consider a more general acceleration function with respect to \cite{Shen, PRV}, for the free-will leader. Indeed we assume
\begin{equation}\label{Q1}
\vert f(t)\vert =o((1+t)^{1-N }) \quad \mbox{\rm and}\quad
 t^{N-2} \vert f(t) \vert \in L^1(0,+\infty)
\end{equation}
instead of
\begin{equation}\label{suf}
\vert f(t)\vert =O((1+t)^{-\mu }),
\quad \mu >N-1\,.
\end{equation}
Then, for instance, $f $ can be in the form
$$f(t)=\frac C {(1+t)^{\mu} },\ \mu >N-1,$$ as in \cite{Shen, PRV}, but also
$$f(t) = \frac C {(1+t)^{N-1}\,\ln^2 (2+t)}\,.$$
Note that, from (\ref{Q1}) it results
\begin{equation}\label{weneed}
t^k\vert f(t)\vert = o( (1+t)^{1-N+k}),\quad \forall \ k=1,\dots, N-1\,.
\end{equation}
In order to prove our flocking result, we will need the following lemma, which is a generalization of
Lemma \ref{usoLyap} above.

\begin{Lemma}\label{usoLyapG}
Let $(x, v)$ be a trajectory in the phase--space, namely $\frac{dx}{dt}(t)=v(t)$ for $t\ge 0\,.$
Assume that
\begin{equation}\label{XX1}
\frac {d\vert v\vert}{dt}(t)\le -d_0\psi (\vert x(t)\vert + M) \vert v(t)\vert + g(t)\quad \forall \ t \ge t_0,
\end{equation}
for some nonnegative constants $M, t_0,$ a constant $d_0>0$ and a continuous and integrable function  $g:[t_0,+\infty)\rightarrow (0,+\infty),$ where $\psi: [0,+\infty)\rightarrow (0,+\infty)$ is a continuous function satisfying $(\ref{divergent}).$
Then, there exists a suitable positive constant $C$ such that
$$\vert x(t)\vert \le C, \quad t\ge 0\,.$$
\end{Lemma}

\begin{proof}
Let us consider the functionals ${\mathcal {F}}_\pm$ introduced in (\ref{PP2}) with $d_0, M, \psi$ as in the statement.
From (\ref{XX1}) we deduce
 \begin{equation}\label{XX3}
\begin{array}{l}
\displaystyle{
\frac{d\mathcal{F}_{\pm}}{dt}(t)=
\frac{d|{v}|}{dt}(t)\pm d_0\psi(|{x(t)}|+M )\frac{d|{x}|}{dt}(t)
}\\
\hspace{1 cm}
\displaystyle{
\leq -d_0\psi(|{x}(t)|+M)|{v(t)}|\pm d_0\psi(|{x(t)}|+M)\frac{d|{x}|}{dt}(t)+g(t)}\\
\hspace{1 cm}\displaystyle{
=d_0\psi(|{x(t)}|+M)\left(
\pm\frac{d|{x}|}{dt}(t) -|{v(t)}| \right )+g(t)\leq g(t)\,,\quad t\ge t_0\,,}
\end{array}
\end{equation}
where we have used inequality (\ref{PP4}).

Now, we integrate (\ref{XX3}) on the time interval $[t_0,t],$ obtaining
$$\mathcal{F}_{\pm}(t)-\mathcal{F}_{\pm}(t_0 )\leq \Vert g\Vert_{L^1(t_0,+\infty)},$$
which gives

$$|{v}(t)|\leq
\pm d_0\left( \phi\left(\vert x (t_0)\vert + {M}\right)-\phi\left(|{x}(t)|+{M}\right)\right)+|{v}(t_0 )|+\Vert g\Vert_{L^1(t_0,+\infty)},$$
namely
\begin{equation}\label{XX5}
 |{v}(t)|\leq -d_0\left|\int_{|{x}(t_0 )|+{M}}^{|{x}(t)|+{M}}\psi(s)\;ds\right| +|{v}(t_0 )|+\Vert g\Vert_{L^1(t_0,+\infty)}\,.
\end{equation}
Therefore, from (\ref{XX5}), we have

\begin{equation}\label{XX6}
	|{v}(t_0 )|+\Vert g\Vert_{L^1(t_0,+\infty)}\geq d_0\left|\int_{|{x}(t_0)|+{M}}^{|{x}(t)|+{M}}\psi(s)\;ds\right|.
\end{equation}

The assumption (\ref{divergent}) ensures then the existence of a constant  $x_M>0$ such that

$$|{v}(t_0)|+\Vert g\Vert_{L^1(t_0,+\infty)}=d_0\int_{|{x}(t_0)|+{M}}^{x_M}\psi(s)\;ds ,$$
which, together with (\ref{XX6}), implies
$|{x}(t)|\leq C,\quad\forall\  t\ge 0\,.$
\end{proof}

\begin{Theorem}\label{FlockFWDistributed}
Let $(x_i, v_i),$ $i=1, \dots, N,$ be a solution of the Cucker-Smale system under hierarchical leadership with delay  $(\ref{agent1})$--$(\ref{CSShendelayFW})$ with initial conditions
$(\ref{ICFW}).$
Assume that $(\ref{divergent})$
is satisfied and that the acceleration of the free--will leader satisfies $(\ref{Q1}).$
Then, it results
\begin{equation}\label{Q2}
\vert v_i(t)-v_j(t)\vert \rightarrow 0,\quad \mbox{\rm for }\ \ t\rightarrow +\infty\,,\quad \forall\ i,j=1,\dots,N\,.
\end{equation}
\end{Theorem}

\noindent
{\sl Proof.}
As in the previous convergence to consensus result,
we argue by induction. First, we look at  the first agent, i.e. the free-will leader.
Equation (\ref{agent1}) gives
$$v_1(t)=v_1(0)+\int_0^t f(s)\, ds\,,$$
and so, from (\ref{acceleration}),
\begin{equation}\label{Q2}
\vert v_1(t)\vert \le\vert v_1(0)\vert +\Vert f\Vert_1=C_1\,,\quad \forall\ t\ge 0\,.
\end{equation}
Now, let us consider the 2-flock. As before, let us denote
$$w_2(t)=v_2(t)-v_1(t) \quad\mbox{\rm and }\quad y_2(t)=x_2(t)-x_1(t),\quad t\ge 0\,.$$
From (\ref{agent1}) and (\ref{CSShendelayFW})
\begin{equation}\label{Q3}
\begin{array}{l}
\displaystyle{\frac {dw_2}{dt}(t)=\frac {dv_2}{dt}(t)-\frac {dv_1}{dt}(t)
=\int_{t-\tau }^t\mu (t-s)\psi_{21}(s)[v_1(s)-v_2(t)]\, ds-f(t)}\\
\hspace{0.6 cm}\displaystyle{
=(v_1(t)-v_2(t)) \int_{t-\tau}^t \mu (t-s)\psi_{21}(s)\, ds -
\int_{t-\tau}^t \mu (t-s) \psi_{21}(s)[v_1(t)-v_1(s)]\, ds -f(t)}
\\
\hspace{0.6 cm} \displaystyle{=- w_2(t) \int_{t-\tau }^t \mu (t-s)\psi_{21}(s )\,ds -\int_{t-\tau }^t\mu (t-s) \psi_{21}(s )\int_{s}^tf(\sigma )\,d\sigma\, ds -f(t)\,,\quad t\ge\tau\,.
}
\end{array}
\end{equation}
Now, from (\ref{Q1}), it results
\begin{equation}\label{Q4}
\begin{array}{l}
\displaystyle{
\left\vert \int_{t-\tau }^t \mu (t-s) \psi_{21}(s)\int_{s}^tf(\sigma )\,d\sigma\, d s\,\right\vert +\vert f (t)\vert } \\\medskip
\hspace{1.5 cm}\displaystyle{
\le
\tau\mu_0 \max_{s\in [0,+\infty )}\psi (s)\,\int_{t-\tau}^t \vert f(s) \vert \,ds +\vert f(t)\vert =O( \vert f\vert )\,.
}
\end{array}
\end{equation}
Then, from (\ref{Q3}) and (\ref{Q4}) we obtain
\begin{equation}\label{Q5}
\frac {d\vert w_2\vert }{dt}(t)\le -\vert
w_2(t)\vert\int_{t-\tau}^t \mu(t-s)\psi_{21}(s)\, ds  + \tilde f(t)\,,\quad t\ge\tau\,.
\end{equation}
where
\begin{equation}\label{tildef}
\tilde f(t):= \tau\mu_0 \max_{s\in [0,+\infty )}\psi (s)\,\int_{t-\tau}^t \vert f(s) \vert \,ds +\vert f(t)\vert =O(\vert f\vert )\,.
\end{equation}
Therefore,
\begin{equation}\label{Q6}
\vert w_2(t)\vert \le \vert w_2(\tau )\vert +\int_{\tau}^{+\infty} \tilde f(t)\, dt\le D_2,\quad \forall\ t\ge \tau\,,
\end{equation}
for some constant $D_2>0.$
Since
$$y_2(s )=y_2(t)+\int_t^{s}w_2(\sigma )\, d\sigma\,,$$
from (\ref{Q6}) we have
\begin{equation}\label{Q7}
\vert y_2(s)\vert \le \vert y_2(t)\vert +\tau D_2\,,\quad \forall \ s\in [t-\tau, t]\,.
\end{equation}
From (\ref{Q5}) and (\ref{Q7}), we then deduce
\begin{equation}\label{Q8}
\frac {d\vert w_2\vert }{dt}(t)\le -\mu_0\psi (\vert y_2(t)\vert +\tau D_2 )\vert
w_2(t)\vert +\tilde f(t)\,, \quad t\ge\tau\,.
\end{equation}

Then, we can apply Lemma \ref{usoLyapG} to the pair $(y_2, w_2)$ with $d=\mu_0, M=\tau D_2$ and $g=\tilde f\,,$
obtaining that

\begin{equation}\label{XX7}
|{y_2}(t )|+\tau D_2 \leq y^2_R,
\quad t\ge 0\,,
\end{equation}
for a suitable positive constant $y^2_R\,.$
So, from (\ref{Q8}) and (\ref{XX7}) we have
$$\frac{d|{w_2}|}{dt}(t)\leq -\psi(y^2_R)|{w_2}(t)|+ \tilde f(t) ,\quad t\ge \tau\,,$$
and thus, for every $T>\tau ,$ applying  Gronwall's lemma we deduce
\begin{equation}\label{XX8}
\begin{array}{l}
\displaystyle{\vert w_2(T)\vert \le e^{-\psi(y^2_R)\frac T2}\vert w_2( T/2 )\vert +\int_{\frac T2}^Te^{-\psi(y^2_R)(T-t)}\tilde f(t)\, dt}\\
\hspace {1 cm} \displaystyle{\le e^{-\psi(x^2_R)\frac T 2}D_2 +\int_{\frac T2}^T\tilde f(t)\, dt
\le e^{-\psi(x^2_R)\frac T 2}D_2 +\tilde f_2(T) \,,}
\end{array}
\end{equation}
where, recalling (\ref{Q1}), $\tilde f_2,$ is a suitable function satisfying
\begin{equation}\label{f2}
\tilde f_2 (t)=O(t\vert f\vert )=o((1+t)^{2-N})\,.
\end{equation}
Thus,
\begin{equation}\label{XX10}
\vert v_2(t)-v_1(t)\vert =o((1+t)^{2-N})\,.
\end{equation}
Note also that
\begin{equation}\label{XX11}
\vert v_1(t-\tau) -v_1(t)\vert \le \int_{t-\tau}^{t} \vert f (t)\vert\, dt =O(\vert f\vert )\,,
\end{equation}
and then
\begin{equation}\label{XX12}
\begin{array}{l}
\displaystyle{
\vert v_2(t-\tau) -v_2(t)\vert
\le \vert v_2(t-\tau ) -v_1(t-\tau )\vert }\\
\hspace{1.8 cm}\displaystyle{
+
\vert v_1(t-\tau ) -v_1(t)\vert +
\vert v_1(t) -v_2(t)\vert =o((1+t)^{2-N})}\,.
\end{array}
\end{equation}
Therefore, (\ref{XX10})--(\ref{XX12}) imply
\begin{equation}\label{2flockFW}
\vert v_i(t-\tau) -v_j(t)\vert =O(\tilde f_2 )= o((1+t)^{2-N}),\quad \mbox{\rm for}\ \ i,j\in \{1, 2\}\,.
\end{equation}

Now, as induction hypothesis, assume that
for a flock of
 $l-1$ agents $[1, \dots, l-1]$ with $2<l\le N$, we have
\begin{align}
	&|{v_i}(t)-{v_j}(t)|=O(t^{l-2} \vert f\vert ) =o((1+t)^{l-1-N}),\label{inductionassumptiondelay1FW}\\
	&|{v_i}(t-\tau)-{v_j}(t)|=O(t^{l-2} \vert f\vert =o((1+t)^{l-1-N})\,,
	\label{inductionassumptiondelay2FW}
\end{align}
for all $i,j\in \{ 1,\dots, l-1\}.$

 Then, we want to prove  the same kind of estimates  for a  flock with $l$ agents.
This will complete our theorem.

As before, we will use the average position and velocity of the leaders of agent $l$, introduced in (\ref{PP19}) and let $y_l, w_l$ be defined as in  (\ref{PP20}).
Then,
as before we can write
\begin{equation}\label{XX14}	
\frac{d {w_l}}{dt}=-w_l(t)\sum_{j\in\mathcal{L}(l)}
\int_{t-\tau }^t \mu (t-s) \psi_{lj}(s)ds +
\sum_{j\in\mathcal{L}(l)}\int_{t-\tau}^t\mu (t-s) \psi_{lj}(s)[{v_j}(s)-{\hat {v_l}}(t)] ds
-\frac{d {\hat{v_l}}}{dt}.
\end{equation}

Using the induction hypothesis (\ref{inductionassumptiondelay2FW}), since $\mathcal{L}(i),\mathcal{L}(l)\subseteq [1, \dots, l-1]$,

\begin{equation}\label{epsilon1delayFW}
	\frac{d {\hat{v_l}}}{dt}=\frac{1}{d_l}\sum_{i\in\mathcal{L}(l)}\frac{d {v_i}}{dt}=\chi_{1\in\mathcal{L}(l)}\frac{1}{d_l}{f}(t)+\frac{1}{d_l}\sum_{i\in\mathcal{L}(l)\backslash \{1\}}\frac{d {v}_i}{dt}=O(t^{l-2}\vert f\vert )=o((1+t)^{l-1-N}).
\end{equation}
From the induction hypotheses  (\ref{inductionassumptiondelay2FW}) we deduce also

\begin{equation}\label{XX15}
\begin{array}{l}
\displaystyle{
	\sum_{j\in\mathcal{L}(l)}\int_{t-\tau}^t \mu (t-s) \psi_{lj}(s)[{v_j}(s)-{\hat{v_l}}(t)]\;ds}
\\
\displaystyle{\hspace{1 cm}
	= \frac 1 {d_l}\sum_{j\in\mathcal{L}(l)}\int_{t-\tau}^t\mu (t-s)\psi_{lj}(s)
	\Big (\sum_{i\in\mathcal{L}(l)} [v_j(s) -v_i(t)]\Big )\;ds
	=O(t^{l-2}\vert f\vert )=o((1+t)^{l-1-N})\,.}
\end{array}
\end{equation}

Then, identity (\ref{XX14}) can be rewritten as

\begin{equation} \label{XX16}
	\frac{d {w_l}}{dt}(t)=- w_l(t)\sum_{j\in\mathcal{L}(l)}\int_{t-\tau}^t\mu (t-s)\psi_{lj}(s)\;ds+ O(t^{l-2}\vert f\vert ),
\quad t\ge\tau\,.	
\end{equation}

As before one can now observe that
for every $j\in\mathcal{L}(l)$ it results
\begin{equation}\label{XX17}
\begin{array}{l}
\displaystyle{
|{x_l}(s)-{x_j}(s)|\le |{x_l}(s)-{\hat{x}_l}(s)\vert +|{x_j}(s)-{\hat{x}_l}(s)|}
\\
\hspace{2.6 cm}\displaystyle{
\leq |{y_l}(s)|+R_l,}
\end{array}
\end{equation}
for some positive $R_l$, due to the induction's assumption. Thus, (\ref{XX16}) implies

\begin{equation}\label{XX18}
	\frac{d|{w_l}|}{dt}(t)\leq -d_l \vert w_l(t)\vert \;\int_{t-\tau}^t \mu (t-s) \psi\left(|{y_l}(s)|+R_l\right)\; ds+O(t^{l-2}\vert f\vert ),\quad t\ge\tau\,.
\end{equation}

Note that (\ref{XX18}) implies

\begin{equation}\label{XX19}
\frac{d|{w_l}|}{dt}\leq O(t^{l-2}\vert f\vert )\,.
\end{equation}
So, recalling the assumptions (\ref{Q1}) on the acceleration $f$ of the free--will leader, we deduce

\begin{equation}\label{XX20}
\vert w_l(t)\vert \le \vert w_l(\tau)\vert +\int_{\tau}^{+\infty }
O(t^{l-2}\vert f\vert )\, dt \le C_l\,.
\end{equation}
Then,
\begin{equation}\label{XX21}
|{x^l}(t-\tau)|\leq |{x^l}(t)|+\int_{t-\tau }^t \vert v^l(s)\vert \, ds \le |{x^l}(t)|+ C_l\tau\,,
\quad t\ge \tau \,,
\end{equation}
which, used in (\ref{XX18}), gives

\begin{equation}\label{XX22}
	\frac{d|{w_l}|}{dt}(t)\leq -d_l \mu_0 \psi\left(|{y_l}(t)|+2\tau C_l+R_l\right )
|{w_l}(t)|+O(t^{l-2}\vert f\vert ).
\end{equation}

We can then apply Lemma \ref{usoLyapG} to the pair state--velocity $(y_l, w_l)$ and conclude  that
$\vert y_l (t)\vert \le C_l$ for some positive constant $C_l.$ So, for a suitable constant $y_M^l,$
$$\vert y_l(t)\vert +2\tau C_l+R_l\le y_M^l,\quad t\ge\tau\,.$$
Using the above estimate in (\ref{XX22}) we then obtain
$$\frac{d|{w_l}|}{dt}\leq -d_l\mu_0\psi(y^l_M)|{w_l}(t)|+ O(t^{l-2}\vert f\vert )\,.$$
Thus, we can apply the Gronwall's lemma analogously  to the  $2-$flock
case obtaining
\begin{equation}\label{XX50}
\vert v^l(t)\vert =O (t^{l-1} \vert f\vert )=o(t^{l-N})\,.
\end{equation}
Then, from (\ref{XX50}) and the induction hypothesis (\ref{inductionassumptiondelay1FW}), for every $j\in {\mathcal L}(l),$ we have

\begin{equation}
	|{v_l}(t)-{v_j}(t)|\le \vert v_l(t)-\hat v_l(t)\vert +
	\vert \hat v_l(t) -v_j(t)\vert = O (t^{l-1} \vert f\vert )=o(t^{l-N}).
	\label{XX51}
\end{equation}

Now, it remains  to prove  that, for all $i,j\in\{1, \dots,l\}$,

\begin{equation}
|{v_i}(t-\tau)-{v_j}(t)|=O (t^{l-1} \vert f\vert )=o(\vert f\vert^{l-N}).
\label{XX52}
\end{equation}

If $i,j\in\{1, \dots, l-1\}$, then (\ref{XX52}) is true by (\ref{inductionassumptiondelay2FW}). Consider the case $i\in\{1, \dots, l-1\}$ and $j=l$. Then,

$$
|{v_i}(t-\tau )-{v_l}(t)|
\leq
|{v_i}(t-\tau )-{v_i}(t)|+\vert {v_i}(t)-{v_l}(t)|=O (t^{l-1} \vert f\vert )=o(\vert f\vert^{l-N}),$$
by (\ref{inductionassumptiondelay2FW}) and (\ref{XX51}).

For the case  $i=j=l,$ using the previous estimates, we obtain

\begin{equation}
\label{XX60}
\begin{array}{l}
\displaystyle{
|{v_l}(s)-{v_l}(t)|=\left|\int_{s}^t {v_l}^\prime (\sigma )\;d\sigma\right|=\left|\int_{s}^t \sum_{k\in\mathcal{L}(l)}\int_{\sigma -\tau }^{\sigma} \mu (\sigma -r) \psi_{lk}(r)\left({{v_k}(r)-{v_l}(\sigma )}\right)\;d r\; d\sigma\right|}\\
\hspace{2 cm}\displaystyle{
\leq C\int_{s}^t O (\sigma ^{l-1} \vert f(\sigma )\vert ) \;d\sigma = O (t^{l-1} \vert f\vert )\,, \quad s\in [t-\tau , t].}
\end{array}
\end{equation}
Also for the last case, where $j\in\{1, \dots, l-1\}$ and $i=l$, using (\ref{XX51}) and (\ref{XX60}) we obtain

$$|{v_l}(t-\tau)-{v_j}(t)|\leq |{v_l}(t-\tau)-{v_l}(t)|+|{v_l}(t)-{v_j}(t)|=O (t^{l-1}\vert  f \vert )=o(\vert f\vert^{l-N})\,.$$
Therefore, (\ref{XX52}) is satisfied for all $i,j\in \{1,\dots,l\}$ and so the theorem is proved.\qed

\begin{Remark}\label{FWmigliora}
{\rm
Note that our generalization concerning the acceleration function $f$ of the free--will leader
is suitable also for the problem without delay considered by Shen \cite{Shen} and for the problem with pointwise delay studied by the authors \cite{PRV}. Therefore, our flocking estimates (\ref{Q2}) could be obtained, under the same assumptions on $f,$ for the problem with free--will leader studied in \cite{Shen} and the more general one considered in \cite{PRV}.
}
\end{Remark}


\begin{thebibliography}{10}

\bibitem{Albi}
G.~Albi, M.~Herty, and L.~Pareschi.
\newblock Kinetic description of optimal control problems and applications to opinion consensus.
\newblock {\em Commun.  Math. Sci.,} 13:1407--1429, 2015.




\bibitem{Bellomo}
N.~Bellomo, M.~A.~Herrero and A.~Tosin.
\newblock On the dynamics of social conflict: Looking for
the Black Swan.
\newblock {\em Kinet. Relat. Models}, 6:459--479, 2013.



\bibitem{Borzi}
A.~Borz\`i and S.~Wongkaew.
\newblock Modeling and control through leadership of a refined flocking system.
\newblock {\em Math. Models Methods Appl. Sci.}, 25:255--282, 2015.





\bibitem{canuto}
C. Canuto, F. Fagnani and P. Tilli.
\newblock An Eulerian approach to the analysis of Krause's consensus models.
\newblock {\em SIAM J. Control Optim.},50:243--265, 2012.

\bibitem{Caponigro1}
M.~Caponigro,~M. Fornasier, B.~Piccoli and E.~Tr\'elat.
\newblock { Sparse stabilization and optimal control of the Cucker-Smale model}.
\newblock {\em Math. Cont. Related Fields,} 3:447--466, 2013.

\bibitem{Caponigro2}
M. Caponigro, M. Fornasier, B. Piccoli and E. Tr\'elat.
\newblock { Sparse stabilization and control of alignment models}.
\newblock {\em Math. Models Methods Appl. Sci.,} 25:521--564, 2015.



\bibitem{carfor}
J.A. Carrillo, M. Fornasier, J. Rosado and G. Toscani.
\newblock { Asymptotic flocking dynamics for the kinetic Cucker-Smale model}.
\newblock {\em SIAM J. Math. Anal.}, 42:218--236, 2010.






\bibitem{Choi}
Y.P.~Choi and J.~Haskovec.
\newblock Cucker-Smale model with normalized communication weights and time delay.
\newblock {\em Kinet. Relat. Models}, 10:1011--1033, 2017.

\bibitem{ChoiLi}
Y.P.~Choi and Z.~Li.
\newblock Emergent behavior of Cucker–Smale flocking particles with heterogeneous time delays.
\newblock {\em Appl. Math. Lett.}, to appear.

\bibitem{Couzin}
I.~Couzin, J.~Krause, N.~Franks and S.~Levin.
\newblock  Effective leadership and decision making in
animal groups on the move.
\newblock {\em Nature}, 433:513--516, 2005.

\bibitem{Cristiani}
E.~Cristiani, B.~Piccoli and A.~Tosin.
\newblock Multiscale modeling of granular flows with application
to crowd dynamics.
\newblock {\em Multiscale Model. Simul.}, 9:155--182, 2011.

\bibitem{CuckerDong}
F..~Cucker and J.G..~Dong.
\newblock A general collision-avoiding flocking framework.
\newblock {\em IEEE Trans. Automat. Cont.},
56:1124--1129, 2011.

\bibitem{CuckerMordecki}
F.~Cucker and E.~Mordecki.
\newblock Flocking in noisy environments.
\newblock {\em J. Math. Pures Appl.}, 89:278--296, 2008.


\bibitem{CS1}
F.~Cucker  and S.~Smale.
\newblock Emergent behaviour in flocks.
\newblock {\em IEEE Transactions on Automatic Control}, 52:852--862, 2007.



\bibitem{CS2}
F.~Cucker  and S.~Smale.
\newblock
On the mathematics of emergence.
\newblock {\em Japanese Journal of Mathematics}, 2:197--227, 2007.

\bibitem{Da}
F.~Dalmao and E.~Mordecki.
\newblock
Cucker-Smale Flocking under Hierarchical Leadership and Random Interactions.
\newblock {\em SIAM J. Appl. Math.}, 71:1307--1316, 2011.

\bibitem{degond}
P.~Degond and S.~Motsch.
\newblock {Continuum limit of self-driven particles with orientation interaction}.
\newblock {\em Math. Models Methods Appl. Sci.}, 18:1193--1215, 2008.

\bibitem{Delay2}
R.~Erban, J.~Haskovec and Y.~Sun.
\newblock
On Cucker-Smale model with noise and delay.
\newblock {\em SIAM J. Appl. Math.}, 76(4):1535--1557, 2016.



\bibitem{HaLee}
S.Y.~Ha, K.~Lee and D.~Levy.
\newblock Emergence of time-asymptotic flocking in a stochastic Cucker-Smale system.
\newblock {\em Commun. Math. Sci.,} 7:453--469, 2009.

\bibitem{HL}
S.Y.~Ha and J.G.~Liu.
\newblock A simple proof of the Cucker-Smale flocking dynamics and mean-field limit.
\newblock {\em Commun. Math. Sci.}, 7:297--325, 2009.

\bibitem{HaSlemrod}
S.Y.~Ha and M.A.~Slemrod.
\newblock Flocking Dynamics of Singularly Perturbed Oscillator Chain and the Cucker-Smale System.
\newblock {\em J. Dyn. Diff. Equat.}, 22:325--330, 2010.

\bibitem{HT}
S.Y.~Ha and E.~Tadmor.
\newblock From particle to kinetic and hydrodynamic descriptions of flocking.
\newblock {\em Kinet. Relat. Models}, 1:415--435, 2008.



\bibitem{Lemercier}
S.~Lemercier, A.~Jelic, R.~Kulpa, J.~Hua, J.~Fehrenbach,
P.~Degond, C.~Appert
Rolland, S.~Donikian and J. Pettr\'{e}.
\newblock Realistic following behaviors for crowd simulation.
\newblock {\em
Comput. Graph. Forum}, 31:489--498, 2012.

\bibitem{Li}
Z.~Li.
\newblock Effectual leadershipin flocks with hierarchy and individual preference.
\newblock {\em Discrete Contin. Dyn. Syst.}, 34:3683--3702, 2014.

\bibitem{LHX}
Z.~Li, S.Y.~Ha and X.~Xue.
\newblock Emergent phenomena in an ensemble of Cucker-Smale particles under joint rooted leadership.
\newblock{\em Math. Models Methods
Appl. Sci.}, 24: 1389--1419, 2014.

\bibitem{LX}
Z.~Li and X.~Xue.
\newblock Cucker-Smale Flocking under Rooted Leadership with Fixed and Switching Topologies.
\newblock {\em SIAM J. Appl. Math.}, 70:3156--3174, 2010.

\bibitem{LiY}
C.H.~Li and S.Y.~Yang.
\newblock A new discrete Cucker-Smale flocking model under hierarchical leadership.
\newblock {\em Discrete Contin. Dyn. Syst. Ser. B}, 21:2587--2599, 2016.

\bibitem{Delay1}
Y.~Liu and J.~Wu.
\newblock Flocking and asymptotic velocity of the Cucker-Smale model with processing delay.
\newblock {\em  J. Math. Anal. Appl.}, 415:53--61, 2014.

\bibitem{Mech}
N.~Mecholsky, E.~Ott and T.~M.~Antonsen.
\newblock  Obstacle and predator avoidance in a model for flocking.
\newblock {\em Phys. D}, 239:988--996, 2010.


\bibitem{MT}
S.~Motsch and E.~Tadmor.
\newblock
A new model for self--organized dynamics and its flocking behavior.
\newblock {\em J. Stat. Phys.}, 144:923--947, 2011.

\bibitem{MT_SIREV}
S.~Motsch and E.~Tadmor.
\newblock Heterophilious Dynamics Enhances Consensus.
\newblock {\em SIAM Rev.} 56:577--621, 2014.





\bibitem{PRT}
B.~Piccoli, F.~Rossi and E.~Tr\'{e}lat.
\newblock Control to flocking of the kinetic Cucker-Smale model.
\newblock{\em
SIAM J. Math. Anal.,}  47:4685--4719, 2015.

\bibitem{PRV}
C.~Pignotti and I.~Reche Vallejo
\newblock{Flocking estimates for the Cucker-Smale model with time lag and hierarchical leadership}.
\newblock {\em J. Math. Anal. Appl.,} 464:1313--1332, 2018.

\bibitem{PT}
C.~Pignotti and E.~Tr\'{e}lat.
\newblock {Convergence to consensus of the general finite-dimensional Cucker-Smale model with time-varying delays}.
\newblock {\em Preprint 2017, ArXiv:1707.05020.}


\bibitem{Shen}
J.~Shen.
\newblock {Cucker-Smale flocking under hierarchical leadership.}
\newblock {\em SIAM J. Appl. Math.}, 68:694--719, 2007/08.


\bibitem{Toscani}
G.~Toscani.
\newblock Kinetic models of opinion formation.
\newblock {\em Commun. Math. Sci.,} 4:481--496, 09 2006.


 \bibitem{Vicsek}
T.~Vicsek, A.~Czirok, E.~ Ben Jacob, I.~ Cohen and O.~Shochet.
\newblock  Novel type of phase
transition in a system of self-driven particles.
\newblock {\em Phys. Rev. Lett.}, 75:1226--1229, 1995.

\bibitem{Wongkaew}
S.~Wongkaew, M.~Caponigro and A.~Borz\`i.
\newblock On the control through leadership of the Hegselmann-Krause opinion formation model.
\newblock {\em Math. Models Methods Appl. Sci.,} 25:565--585, 2015.

\bibitem{Yates}
C.~Yates, R.~Erban, C.~Escudero, L.~Couzin, J.~Buhl, L.~Kevrekidis, P.~Maini and
D.~Sumpter.
\newblock Inherent noise can facilitate coherence in collective swarm motion.
\newblock {\em Proc.
Natl. Acad. Sci. USA}, 106:5464--5469, 2009.

\end{thebibliography}
\end{document}